\documentclass[11pt]{article}
\usepackage{amsmath, amsthm, amscd, amsfonts, amssymb, graphicx, color}
\usepackage[bookmarksnumbered, colorlinks, plainpages]{hyperref}
\usepackage[utf8]{inputenc}
\newtheorem{thm}{Theorem}[section]
\newtheorem{cor}[thm]{Corollary}
\newtheorem{lem}[thm]{Lemma}
\newtheorem{prop}[thm]{Proposition}
\newtheorem{defn}[thm]{Definition}

\numberwithin{equation}{section}

\newcommand{\be}{\begin{equation}}
\newcommand{\ee}{\end{equation}}
\newcommand{\ben}{\begin{enumerate}}
\newcommand{\een}{\end{enumerate}}
\newcommand{\beq}{\begin{eqnarray}}
\newcommand{\eeq}{\end{eqnarray}}
\newcommand{\beqn}{\begin{eqnarray*}}
\newcommand{\eeqn}{\end{eqnarray*}}

\newcommand{\bpf}{\begin{proof}}
\newcommand{\epf}{\end{proof}}
\newcommand{\bl}{\begin{lem}}
\newcommand{\el}{\end{lem}}
\newcommand{\bp}{\begin{prop}}
\newcommand{\ep}{\end{prop}}
\newcommand{\bd}{\begin{defn}}
\newcommand{\ed}{\end{defn}}
\newcommand{\bt}{\begin{thm}}
\newcommand{\et}{\end{thm}}

\title {Convergence of Finslerian metrics under Ricci flow}
\author { M. Yar Ahmadi and B. Bidabad\footnote{The corresponding author}}
\date{\small Faculty of Mathematics and Computer Science,\\ Amirkabir University of Technology,\\ Tehran 15914, Iran.\\
m.yarahmadi@aut.ac.ir\\bidabad@aut.ac.ir}

\begin{document}
\maketitle
\begin{abstract}
In this work, convergence of evolving Finslerian metrics first in a general flow next under Finslerian Ricci flow is studied. More intuitively it is proved that a family of Finslerian metrics $g(t)$ which are solutions to the Finslerian Ricci flow converge in $C^{\infty}$ to a smooth limit Finslerian metric as $ t $ approaches the finite time $ T $. As a consequence of this result one can show that  in a compact  Finsler manifold the curvature tensor along Ricci flow blows up in short time.
\end{abstract}
{\bf Keywords:} Finsler geometry, Ricci flow, convergence in $C^{\infty}$, blow up.\\
\textbf{MSC(2010)}: Primary: 53C60; Secondary: 53C44,
35C08.
\section{Introduction}
Ricci flow is a branch of  general  geometric flows, which is an evolution equation for a Riemannian metric in the set of all Riemannian metrics defined on a manifold. Geometric flow can be used to deform an arbitrary metric into an informative metric, from which one can determine the topology of the underlying manifold and hence innovate numerous progress in the proof of some geometric conjectures. In 1982 Hamilton introduced the notion of Ricci flow on Riemannian manifolds by the evolution equation
\begin{align}\label{RRF}
\frac{\partial}{\partial t}g_{ij}=-2Ric_{ij}, \quad g(t=0):= g_0.
\end{align}
The Ricci flow, which evolves a Riemannian metric by its Ricci curvature is a natural analogue of the heat equation for metrics. In Hamilton's celebrated paper \cite{Ha}, it is shown that there is a unique solution to the Ricci flow for an arbitrary smooth Riemannian metric on a closed manifold over a sufficiently short time. The limiting solution obtained from these gives information about
the structure of the singularity. The idea is to consider shorter and shorter time intervals leading up to a singularity
of the Ricci fow, and to rescale the solution on each of these time intervals to obtain solutions on long time intervals with uniformly bounded curvature. In $ 1989 $ W.X. Shi has proved convergence of evolving metrics under Ricci flow, \cite{Shi}. Next he established estimates for the covariant derivatives of the curvature tensor on complete Riemannian manifolds. By using these estimates R. Hamilton \cite{Ha} has shown that any solution to the Ricci flow that develops a singularity in finite time must have unbounded curvature tensor. In fact, he proved the long time existence theorem for the Ricci flow as long as its curvature remains bounded. Also, N. Sesum has shown that any solution to the Ricci flow that develops a singularity in finite time must have unbounded Ricci curvature \cite{Sesum}.

 The concept of Ricci flow  on Finsler manifolds is defined first by D. Bao, cf., \cite{Ba}, choosing the Ricci tensor introduced by H. Akbar-Zadeh,  \cite{Ak3}. It seems to the present authors that, this choice of D. Bao for definition of Ricci tensor, is completely suitable for definition of Ricci flow in Finsler geometry. In fact, in order to define the concept of Ricci tensor,  Akbar-Zadeh has  used Einstein-Hilbert's functional in general relativity, although it  has some computation negligence and
 introduced definition of Einstein-Finsler spaces as critical points of this functional, similar to the Hamilton's work.

 In the present work, we prove that a family of Finslerian metrics $g(t)$ which are solution to the geometric flow $\frac{\partial}{\partial t}g(t)=\omega(t) $ converges in $C^{\infty}$ to a smooth limit Finslerian metric $ \bar{g} $ as $ t $ approaches $ T $.  Next, as an application in the special case, it is proved that on a compact manifold a family of solutions to the Finslerian Ricci flow $\frac{\partial}{\partial t}g_{_{jk}}{\footnotesize_{(t)}}=-2Ric_{jk}$ converges in $C^{\infty}$ to a smooth limit Finslerian metric $ \bar{g} $ as $ t $ approaches $ T $. This result may be used to show that in a  compact Finsler manifold Ricci flow cannot develop a singularity in finite time unless the $ hh- $curvature is bounded. These tools are indispensable for development of Ricci flow on Finsler geometry, specially on determination of interior estimation of some solutions.

\section{Preliminaries and terminologies}
  Let $M$ be a connected differentiable manifold of dimension $n$. Denote the bundle of tangent vectors of $M$ by
$p:TM\longrightarrow M$, the fiber bundle of non-zero tangent
vectors of $M$ by $\pi:TM_0\longrightarrow M$ and the
pulled-back tangent bundle by $\pi^*TM\longrightarrow TM_0$. A point of $TM_0$ is denoted
by $z=(x,y)$, where $x=\pi z\in M$ and $y\in T_{\pi z}M$. Let $(x^i)$ be a local chart with the domain $U\subseteq M$ and $(x^i,y^i)$ the induced local coordinates on $\pi^{-1}(U)$, where ${\bf y}=y^i\frac{\partial}{\partial x^i}\in T_{\pi z}M$, and $i$ running over the range $1,2,...,n$. A (globally defined) \emph{Finsler structure} on $M$ is a function $F:TM\longrightarrow [0,\infty)$ with the following properties;
 $F$ is $C^\infty$ on the entire slit tangent bundle $TM \backslash 0 $;
 $F(x,\lambda y )=\lambda F(x,y) ~ \forall \lambda >0$;
the $n\times n$ Hessian matrix $(g_{ij}) =\frac {1}{2} ([F^2]_{y^i y^j})$ is positive-definite at every point of $TM_0$.
The pair $(M,g)$ is called a {\it Finsler manifold}. Denote by  $TTM_0$ and $SM$ the tangent bundle of $TM_0$ and the sphere bundle respectively, where $SM:=\bigcup\limits _{x\in M} S_xM$ and $S_xM:=\{y\in T_xM | F(y)=1\}$.
Given the induced coordinates $(x^i,y^i)$ on $TM$,  coefficients of spray vector field are defined by $G^{i}=1/4 g^{ih}(\frac{\partial^{2}F^{2}}{\partial y^{h}\partial x^{j}}y^{j}-\frac{\partial F^{2}}{\partial x^{h}})$.
 One can observe that the pair $\{\delta/\delta x^i,\partial/\partial y^i  \}$ forms a horizontal and vertical frame for $TTM$,  where ${\frac {\delta}{\delta x^i}}:={\frac {\partial}{\partial x^i}}-G_i^j{\frac {\partial}{\partial y^j}}$, $G_i^j :={\frac {\partial G^j}{\partial y^i}}$.
There is a canonical linear mapping $\varrho:TTM_0\longrightarrow \pi^*TM,$ where, $\varrho=\pi_*$, $\varrho _z ((\frac {\delta}{\delta x^i})_z)=(\frac {\partial}{\partial x^i})_z$ and $\varrho ((\frac {\partial}{\partial y^i})_z)=0.$
  Let $V_zTM$ be the set of vertical vectors at $z\in TM_0$, that is, the set of vectors which are tangent to the fiber through $z$. Equivalently, $V_zTM=\ker \pi_*$
where $\pi_*:TTM_0\longrightarrow TM$ is the linear tangent mapping.
Let $\tilde{\nabla}$ be a linear connection on $\Gamma(\pi^*TM)$ the sections of pull back bundle $\pi^*TM$,
$$\tilde{\nabla}: T_zTM_0\times\Gamma(\pi^*TM)\longrightarrow\Gamma(\pi^*TM).$$
 There is a linear mapping
$\mu:TTM_0\longrightarrow \pi^*TM,$ defined by $\mu(\hat{X})=\tilde{\nabla}_{\hat{X}}{ v}$ where, $\hat{X}\in TTM_0$ and ${ v}$ is the canonical section of $\pi^*TM$.
The connection $\tilde{\nabla}$ is said to be {\it regular}, if $\mu$ defines an isomorphism between $VTM_0$ and
$\pi^*TM$. In this case, there is a horizontal distribution $HTM$ such that we have the Whitney sum
$
TTM_0=HTM\oplus VTM.
$
It can be shown that the set $\{\frac{\delta}{\delta x^j}\}$ and $\{\frac{\partial}{\partial y^j}\}$, forms a local frame field for the horizontal and vertical subspaces, respectively.
This decomposition permits to rewrite a vector field $\hat{X}\in TTM_0$, uniquely into the form $\hat{X}=H\hat{X}+V\hat{X}$. We will denote in the sequel all
the sections of $\pi^*TM$ by $X=\varrho(\hat X)$, $Y=\varrho(\hat Y)$,
 and the corresponding complete lift on $TTM_0$ by $\hat{X}, \hat{Y}$
respectively, unless otherwise specified.

 The {\it torsion} and {\it curvature} tensors of the regular connection $\tilde{\nabla}$ are given by
 \begin{eqnarray*}
\tau(\hat{X},\hat{Y})&=&\tilde{\nabla}_{\hat{X}}Y-\tilde{\nabla}_{\hat{Y}}X-\varrho[\hat{X},\hat{Y}],\\
\Omega(\hat{X},\hat{Y})Z&=&\tilde{\nabla}_{\hat{X}}\tilde{\nabla}_{\hat{Y}}Z-\tilde{\nabla}_{\hat{Y}}\tilde{\nabla}_{\hat{X}}Z
-\tilde{\nabla}_{[\hat{X},\hat{Y}]}Z,
\end{eqnarray*}
where, $X=\varrho(\hat{X})$, $Y=\varrho(\hat{Y})$, $Z=\varrho(\hat{Z})$ and $\hat{X}$, $\hat{Y}$ and $\hat{Y}$ are vector fields on $TM_0$.  They
determine two torsion tensors denoted here by $S$ and $T$ and three
curvature tensors denoted by $R$, $P$ and $Q$, defined by:
\begin{eqnarray*}
S(X,Y)&=&\tau(H\hat{X},H\hat{Y}),\ \ \ T(\dot{X},Y)=\tau(V\hat{X},H\hat{Y}),\\
R(X,Y)&=&\Omega(H\hat{X},H\hat{Y}),\ \ \ P(X,\dot{Y})=\Omega(H\hat{X},V\hat{Y}),\\
Q(\dot{X},\dot{Y})&=&\Omega(V\hat{X},V\hat{Y}),
\end{eqnarray*}
where, $X=\varrho(\hat{X})$,\ $Y=\varrho(\hat{Y})$,\
$\dot{X}=\mu(\hat{X})$ and $\dot{Y}=\mu(\hat{Y})$. The tensors $R$, $P$ and $Q$ are called $hh-$, $hv-$ and $vv-$curvature tensors, respectively. There is a unique metric compatible $h-$torsion free regular connection $\tilde{\nabla}$ associated to the Finsler structure $F$ satisfying, $\tilde{\nabla}_{\hat{Z}}g=0,$ $S(X,Y)=0,$ and $g(\tau(V\hat{X},\hat{Y}),Z)=g(\tau(V\hat{X},\hat{Z}),Y)$ called \emph{Cartan connection}. Here we denote by $ \nabla $ and $ \dot{\nabla} $, horizontal and vertical covariant derivatives of Cartan connection. We need the following properties of Cartan connection in the sequel, cf., \cite{AZ}.
\begin{align}\label{Eq;gogo}
2g(\nabla_{H\hat{X}}Y,Z)=&H\hat{X}.g(Y,Z)+H\hat{Y}.g(X,Z)-H\hat{Z}.g(X,Y)
\\&+g(\varrho[H\hat{X},H\hat{Y}],Z)+g(\varrho[H\hat{Z},H\hat{X}],Y)+g(\varrho[H\hat{Z},H\hat{Y}],X).\nonumber
\end{align}
For an arbitrary $ (0,2)- $tensor field $ S $ on $ \pi^*TM $ we have
\begin{align}\label{Def;derivative}
(\nabla_{H\hat{X}}S)(Y,Z)=H\hat{X}.S(Y,Z)-S(\nabla_{H\hat{X}}Y,Z)-S(Y,\nabla_{H\hat{X}}Z).
\end{align}
We consider also the \emph{reduced curvature tensor} $R^i_k$ which is expressed entirely in terms of the $x$ and $y$ derivatives of spray coefficients $G^i$, cf. \cite{Sh3}.
 \begin{align}\label{E,Ricci scalar}
 R^{i}_{k}:=\frac{1}{F^2}\big(2\frac{\partial G^i}{\partial x^k}-\frac{\partial^2 G^i}{\partial x^j \partial y^k}y^j +2G^j\frac{\partial^2 G^i}{\partial y^j \partial y^k} - \frac{\partial G^i}{\partial y^j}\frac{\partial G^j}{\partial y^k}\big).
\end{align}
In the general Finslerian setting, one of the remarkable definitions of Ricci tensors is  introduced by H. Akbar-Zadeh \cite{AZ} as follows.
\begin{equation*}
{Ric_{jk} :=[\frac{1}{2} F^{2} \mathcal{R}ic]_{y^{j} y^{k}}},
\end{equation*}
where
$
\mathcal{R}ic=R^{i}_{i}
$
and  $R^i_k$ is defined by (\ref{E,Ricci scalar}). As mentioned above the definition of Einstein-Finsler space related to this Ricci tensor is obtained as critical point of Einstein-Hilbert functional and hence suitable for definition of Finslerian Ricci flow. One of the advantages of the Ricci  quantity defined here  is its independence to the choice of the Cartan, Berwald or Chern(Rund) connections. Based on the Akbar-Zadeh's Ricci tensor, in analogy with the equation (\ref{RRF}), D. Bao has considered, the following natural extension of \emph{Ricci flow} in Finsler geometry, cf.  \cite{Ba},
\begin{align}\label{FRF}
\frac{\partial}{\partial t} g_{jk}=-2Ric_{jk}, \quad {g(t=0):=g_{_0}}.
\end{align}
This equation is equivalent to the following differential equation
\begin{align*}
{\frac{\partial}{\partial t}(\log F(t))=-\mathcal{R}ic,}\quad {F(t=0):=F_{_0}},
\end{align*}
where, $F_{_0}$ is the initial Finsler structure.
\section{Convergence of evolving Finslerian metrics}
Let $M$ be a compact differential manifold, $ F(t) $ a family of smooth $ 1- $parameter Finsler structures on $ M $ and $ g(t) $ the Hessian matrix of $ F(t) $ which defines a scalar product on $ \pi^*TM $ for each $ t $. Suppose that
\begin{align}\label{evolving}
\frac{\partial}{\partial t}g(t)=\omega(t), \quad g(0):=g_0,
\end{align}
where $ \omega(t):=\omega(t,x,y) $ is a family of symmetric $(0,2)$-tensors on $ \pi^*TM $, zero-homogenous with respect to $y$. For a $(p,q)$-tensor field $ \Omega $ define
\begin{align*}
|\Omega|^2_{_{g}}:=\Omega_{i_1...i_p}^{j_1...j_q}\Omega_{j_1...j_q}^{i_1...i_p}=g^{j_1l_1}\cdots g^{j_ql_q}g_{_{i_1k_1}}\cdots g_{_{i_pk_p}}\Omega_{l_1...l_q}^{k_1...k_p}\Omega_{j_1...j_q}^{i_1...i_p}.
\end{align*}
Let $u_0:[0,T)\longrightarrow [0,\infty)$ be a real function defined by
$u_0(t):=\sup\limits_{SM}|\omega(t)|_{_{g(t)}}$, for all $t\in [0,T)$ we have the following lemma.
\begin{lem}\label{A.1}
If $\int_0^Tu_0(t)dt<\infty $, then the Finslerian metrics $ g(t) $ are uniformly equivalent; that is, there exists a positive constant $C$ such that
$$\frac{1}{C}|v|^2_{_{g_{0}}}\le |v|^2_{_{g(t)}}\le C|v|^2_{_{g_{0}}},$$
for all points $((x,y),t)\in SM\times [0,T)$ and all vectors $v\in T_xM$.
\end{lem}
\begin{proof}
Fix a point $((x,y),t)\in SM\times [0,T)$ and a non-zero vector $v\in T_xM$. Clearly, the assertion holds for $ v=0 $. Then by means of Cauchy-Schwarz inequality
\begin{align*}
|\frac{d}{dt}|v|^2_{_{g(t)}}|&=|\frac{d}{dt}(v^iv_i)|=|\frac{d}{dt}(g_{ij}v^iv^j)|=|\omega_{ij}v^iv^j|\le (\omega_{ij}\omega^{ij})^{\frac{1}{2}}(v^iv^jv_iv_j)^{\frac{1}{2}}\\&=(|\omega(t)|^2_{_{g(t)}})^{\frac{1}{2}}(|v|^2_{_{g(t)}}|v|^2_{_{g(t)}})^{\frac{1}{2}}= |\omega(t)|_{_{g(t)}}|v|^2_{_{g(t)}}\le u_0(t)|v|^2_{_{g(t)}}.
\end{align*}
Therefore
\begin{align*}
-u_0(t)|v|^2_{_{g(t)}}\leq\frac{d}{dt}|v|^2_{_{g(t)}}\leq u_0(t)|v|^2_{_{g(t)}},
\end{align*}
and we have
$-u_0(t)\leq\frac{\frac{d}{dt}|v|^2_{_{g(t)}}}{|v|^2_{_{g(t)}}}\leq u_0(t)$. We conclude that
\begin{align*}
-u_0(t)\leq\frac{d}{dt}\ln |v|^2_{_{g(t)}}\leq u_0(t).
\end{align*}
By integration we have
\begin{align*}
-\int_0^Tu_0(\tau)d\tau\leq-\int_0^tu_0(\tau)d\tau\leq \ln |v|^2_{_{g(t)}}-\ln |v|^2_{_{g(0)}}\leq \int_0^tu_0(\tau)d\tau\leq \int_0^Tu_0(\tau)d\tau.
\end{align*}
Thus,
\begin{align*}
-\int_0^Tu_0(\tau)d\tau\leq \ln \frac{|v|^2_{_{g(t)}}}{|v|^2_{_{g(0)}}}\leq \int_0^Tu_0(\tau)d\tau.
\end{align*}
Finally, we have
\begin{align*}
|v|^2_{_{g(0)}}e^{-\int_0^Tu_0(\tau)d\tau}\leq |v|^2_{_{g(t)}}\leq |v|^2_{_{g(0)}}e^{\int_0^Tu_0(\tau)d\tau}.
\end{align*}
By means of the assumption $\int_0^Tu_0(t)dt<\infty$ we put $ C= e^{\int_0^Tu_0(\tau)d\tau}$. This completes the proof.
\end{proof}
 The following proposition allows us to relate the time dependent horizontal covariant derivative of Cartan connection $\nabla$ to a fixed background Finsler connection $D$ defined on $ \pi^{*}TM $.
\begin{prop}\label{A.2}
Let (M,g) be a Finsler space, $\nabla$ the horizontal covariant derivative of Cartan connection associated to the metric $ g $. Moreover, assume that $D$ is a fixed horizontal covariant derivative of a torsion free linear connection on $ \pi^*TM $. Then,
\begin{align*}
\nabla_{H\hat{X}}Y={D}_{H\hat{X}}Y+\Gamma{(H\hat{X},H\hat{Y})},
\end{align*}
where $ \Gamma{(H\hat{X},H\hat{Y})} $ is a symmetric $ (0,2)- $tensor field defined by
\begin{align*}
2g(\Gamma{(H\hat{X},H\hat{Y})},Z)=({D}_{H\hat{X}}g)(Y,Z)+
({D}_{H\hat{Y}}g)(X,Z)-({D}_{H\hat{Z}}g)(X,Y).
\end{align*}
\begin{proof}
Torsion freeness of $ {D} $ reads $ \varrho[\hat{X},\hat{Y}]={D}_{\hat{X}}Y-{D}_{\hat{Y}}X $. Replacing this relation in (\ref{Eq;gogo}) yields
\begin{align*}
2g(\nabla_{H\hat{X}}Y,Z)=&H\hat{X}.g(Y,Z)+H\hat{Y}.g(X,Z)-H\hat{Z}.g(X,Y)\\&+g({D}_{H\hat{X}}Y-{D}_{H\hat{Y}}X,Z)+g({D}_{H\hat{Z}}X-{D}_{H\hat{X}}Z,Y)\\&
+g({D}_{H\hat{Z}}Y-{D}_{H\hat{Y}}Z,X).
\end{align*}
By simplification and making use of (\ref{Def;derivative}) we get
\begin{align*}
2g(\nabla_{H\hat{X}}Y,Z)=&(D_{H\hat{X}}g)(Y,Z)+
(D_{H\hat{Y}}g)(X,Z)\\&-(D_{H\hat{Z}}g)(X,Y)+2g(D_{H\hat{X}}Y,Z).
\end{align*}
 Using definition of $ \Gamma{(H\hat{X},H\hat{Y})} $ we conclude that
\begin{align*}
2g(\nabla_{H\hat{X}}Y,Z)=&2g(\Gamma{(H\hat{X},H\hat{Y})},Z)+2g(D_{H\hat{X}}Y,Z).
\end{align*}
Finally we have
\begin{align*}
g(\nabla_{H\hat{X}}Y,Z)=&g(\Gamma{(H\hat{X},H\hat{Y})}+D_{H\hat{X}}Y,Z).
\end{align*}
This equality holds for the arbitrary section $ Z $ of $ \pi^*TM $, hence
\begin{align*}
\nabla_{H\hat{X}}Y=D_{H\hat{X}}Y+\Gamma{(H\hat{X},H\hat{Y})}.
\end{align*}
As we have claimed.
\end{proof}
 \end{prop}
 Here, we denote by $\nabla^mA$ and $D^mA$ the $ m^{th} $ order iterated horizontal Cartan covariant derivative of the tensor $ A $. Let $A$ and $B$ be two tensor fields defined on $ \pi^*TM $. We denote by $A\ast B$ any linear combination of these tensors obtained by the tensor product $ A\otimes B $ and any of the following operations;\\ I. summation over pairs of matching upper and lower indices;\\ II. contraction on upper indices with respect to the metric $ g $;\\ III. contraction on lower indices with respect to the inverse metric of $ g $.
 \begin{lem}\label{A.3}
Let $\nabla$ denote horizontal covariant derivative in Cartan connection associated to the metric $g(t)$ and $D$ a fixed background torsion free connection. For any positive integer $m$ we have
\begin{align*}
\nabla^m\omega(t)-D^m\omega(t)=&\sum_{l=0}^{m-1}\sum_{ i_1+...+i_q=m-l}D^{i_1}g(t)\ast ...\ast D^{i_q}g(t)\ast D^l\omega(t).
\end{align*}
.
\begin{proof}
The proof is by induction on $m$. By means of (\ref{Def;derivative}) and the Proposition \ref{A.2} we get
\begin{align*}
\nabla_{H\hat{X}}(\omega(t))(Y,Z)-&D_{H\hat{X}}(\omega(t))(Y,Z)\\&=-\omega(t)(\nabla_{H\hat{X}}Y-D_{H\hat{X}}Y,Z)-\omega(t)(Y,\nabla_{H\hat{X}}Z-D_{H\hat{X}}Z)\\
&=-\omega(t)(\Gamma(H\hat{X},H\hat{Y}),Z)-\omega(t)(Y,\Gamma(H\hat{X},H\hat{Z})).
\end{align*}
By definition of $ \Gamma(H\hat{X},H\hat{Y}) $ we obtain
$$\nabla\omega(t)-D\omega(t)=Dg(t)\ast\omega(t).$$
Hence, the assertion holds for $m=1$. Now assume that $m\geq 2$ and the following relation holds.
\begin{align*}
\nabla^{m-1}\omega(t)-D^{m-1}\omega(t)&=\\&\hspace{-1.7cm}\sum_{l=0}^{m-2}\sum_{i_1+...+i_q=m-l-1}D^{i_1}g(t)\ast...\ast D^{i_q}g(t)\ast D^l\omega(t).
\end{align*}
Horizontal covariant derivative of this equation implies
\begin{align}\label{AAA}
\nabla\nabla^{m-1}\omega(t)-\nabla D^{m-1}\omega(t)=&\\&\nonumber\hspace{-2.7cm}\sum_{l=0}^{m-2}\sum_{i_1+...+i_q=m-l-1}D^{i_1}g(t)\ast...\ast D^{i_q}g(t)\ast \nabla D^l\omega(t)\\&\nonumber\hspace{-3cm}+\sum_{l=0}^{m-2}\sum_{i_1+...+i_q=m-l-1}D^{i_1}g(t)\ast ...\ast \nabla D^{i_q}g(t)\ast D^l\omega(t).
\end{align}
Moreover, it follows from Proposition \ref{A.2} that
\begin{align}\label{AAB}
\nabla D^l\omega(t)=DD^l\omega(t)+Dg(t)\ast D^l\omega(t),
\end{align}
and
\begin{align}\label{AAC}
\nabla D^jg(t)=DD^jg(t)+Dg(t)\ast D^jg(t).
\end{align}
 Replacing (\ref{AAB}) and (\ref{AAC}) in (\ref{AAA}), leads
\begin{align*}
\nabla^m\omega(t)-D^m\omega(t)&=\sum_{l=0}^{m-1}\sum_{i_1+...+i_q=m-l}D^{i_1}g(t)\ast ...\ast D^{i_q}g(t)\ast D^l\omega(t).
\end{align*}
As we have claimed.
\end{proof}
 \end{lem}
For simplicity, let $ D $ be a horizontal Cartan covariant derivative associated to the metric $ g(0) $. For all integers $m\geq 1$, we define the continuous functions $u_m:[0,T)\longrightarrow \mathbb{R}$ and $\hat{u}_m:[0,T)\longrightarrow \mathbb{R}$ by
\begin{align}\label{Def;sup}
u_m(t):=\sup \limits_{SM} \mid \nabla^m\omega(t)\mid_{_{g(t)}},\qquad  \hat{u}_m(t):=\sup \limits_{SM} \mid D^m\omega(t)\mid_{_{g(0)}},
\end{align}
 for each $t\in[0,T)$.
By integration of $ \frac{\partial}{\partial t}g(t)=\omega(t) $ we have
 \begin{align*}
g(t)=\int_0^{t}\omega(\tau)d\tau+C.
 \end{align*}
 Putting $ t=0 $ in this relation, yields $ g(0)=C $. We conclude that
 \begin{align}\label{Evol}
g(t)=g(0)+\int_0^t \omega (\tau)d\tau.
 \end{align}
 Now by using the last equation and $Dg(0)=0$, we have
 \begin{align*}
 |D^mg(t)|_{_{g(0)}}=|\int_0^tD^m\omega(\tau)d\tau|_{_{g(0)}}\leq \int_0^t|D^m\omega(\tau)|_{_{g(0)}}d\tau.
 \end{align*}
Therefore
\begin{align*}
 \sup\limits_{SM}|D^mg(t)|_{_{g(0)}}\leq \int_0^t\sup\limits_{SM}|D^m\omega(\tau)|_{_{g(0)}}d\tau.
 \end{align*}
Finally, by means of (\ref{Def;sup}) we conclude that
 \begin{align}\label{109}
 \sup\limits_{SM}|D^mg(t)|_{_{g(0)}}\leq \int_0^t \hat{u}_m(\tau)d\tau,
 \end{align}
 for all $ t\in[0,T) $.
\begin{lem}\label{Last Lemma}
If $\int_0^Tu_m(t)dt<\infty$, for $m=0,1,2,...$, then $\int_0^T\hat{u}_m(t)dt<\infty$ for
$m=0,1,2,...$\hspace{0.09cm}.
\begin{proof}
By Lemma \ref{A.1}, there exists a positive constant $ A $, such that $ |\omega(t)|_{g_0}\leq A|\omega(t)|_{g(t)} $. By assumption, we have $$\int_0^T\hat{u}_0(t)dt =\int_0^T\sup\limits_{SM}|\omega(t)|_{g_0}dt\leq A\int_0^T\sup\limits_{SM}|\omega(t)|_{g(t)}dt=A\int_0^Tu_0(t)dt<\infty.$$
Therefore the assertion holds for $ m=0 $. The proof is by induction on $m$. Assume for a fix integer $ m\geq 1 $ we have $\int_0^T\hat{u}_l(t)dt<\infty$, for
$l=1,2,...,m-1$. It follows from (\ref{109}) that
$\sup\limits_{t\in [0,T)}\sup\limits_{SM}| D^lg(t)|_{_{g(0)}}<\infty$, for $ l=1,2,...,m-1 $. Moreover, by means of Lemma \ref{A.1} the metrics g(t) are uniformly equivalent. Using Lemma \ref{A.3}, we obtain
\begin{align}\label{ineq}
| D^m\omega(t)|_{_{g(0)}}-&| \nabla^m\omega(t)|_{_{g(0)}}
\leq\\&\qquad C_1\sum_{l=0}^{m-1}\sum_{i_1+...+i_q=m-l}| D^{i_1}g(t)|_{_{g(0)}}...| D^{i_q}g(t)|_{_{g(0)}}| D^l\omega(t)|_{_{g(0)}}\nonumber\\&=C_1\sum_{l=1}^{m-1}\sum_{i_1+...+i_q=m-l}| D^{i_1}g(t)|_{_{g(0)}}...| D^{i_q}g(t)|_{_{g(0)}}| D^l\omega(t)|_{_{g(0)}}\nonumber\\&\quad+C_1\sum_{i_1+...+i_q=m}| D^{i_1}g(t)|_{_{g(0)}}...| D^{i_q}g(t)|_{_{g(0)}}| D^l\omega(t)|_{_{g(0)}}\nonumber,
\end{align}
for a positive constant $ C_1 $. By means of induction's assumption and (\ref{109}), there exists a positive constant $ C_2 $ such that
\begin{align}\label{ineq1}
C_1\sum_{l=1}^{m-1}\sum_{i_1+...+i_q=m-l}| D^{i_1}g(t)|_{_{g(0)}}...| D^{i_q}g(t)|_{_{g(0)}}| D^l\omega(t)|_{_{g(0)}}\leq C_2\sum\limits_{l=1}^{m-1}| D^l\omega(t)|_{_{g(0)}},
\end{align}
and
\begin{align}\label{ineq2}
C_1\sum_{i_1+...+i_q=m}| D^{i_1}g(t)|_{_{g(0)}}...| D^{i_q}g(t)|_{_{g(0)}}| D^l\omega(t)|_{_{g(0)}}\leq C_2\big(1+| D^{m}g(t)|_{_{g(0)}}\big)|\omega(t)|_{_{g(0)}}.
\end{align}
Replacing (\ref{ineq1}) and (\ref{ineq2}) in (\ref{ineq}) we have
\begin{align}\label{BBB}
| D^m\omega(t)|_{_{g(0)}}\leq &| \nabla^m\omega(t)|_{_{g(0)}}+C_2\sum\limits_{l=1}^{m-1}| D^l\omega(t)|_{_{g(0)}}\nonumber\\&+C_2\big(1+| D^{m}g(t)|_{_{g(0)}}\big)|\omega(t)|_{_{g(0)}}.
\end{align}
On the other hand by Lemma \ref{A.1}, we conclude that there exists a positive constant $ C_3 $ such that $ | \nabla^m\omega(t)|_{_{g(0)}} \leq C_3|\nabla^m\omega(t)|_{_{g(t)}}$ and $ |\omega(t)|_{_{g(0)}}\leq C_3|\omega(t)|_{_{g(t)}} $. Therefore the equation (\ref{BBB}) reads
\begin{align*}
| D^m\omega(t)|_{_{g(0)}}&\leq C_3| \nabla^m\omega(t)|_{_{g(t)}}+C_2\sum\limits_{l=1}^{m-1}| D^l\omega(t)|_{_{g(0)}}\\&\quad+
C_2C_3\big(1+| D^{m}g(t)|_{_{g(0)}}\big)|\omega(t)|_{_{g(t)}}\\&\leq C_3\sup\limits_{SM}| \nabla^m\omega(t)|_{_{g(t)}}+C_2\sum\limits_{l=1}^{m-1}\sup\limits_{SM}| D^l\omega(t)|_{_{g(0)}}\\&\quad+C_2C_3\big(1+\sup\limits_{SM}| D^{m}g(t)|_{_{g(0)}}\big)\sup\limits_{SM}|\omega(t)|_{_{g(t)}}.
\end{align*}
Using (\ref{Def;sup}) and (\ref{109}), we have
\begin{align*}
\hat{u}_m(t)\leq C_3u_m(t)+C_2\sum\limits_{l=1}^{m-1}\hat{u}_l(t)
+C_2C_3u_0(t)\Big(1+\int_0^t\hat{u}_m(\tau)d\tau\Big),
\end{align*}
for all $ t\in[0,T) $. This implies
\begin{align}\label{CCC}
\frac{\hat{u}_m(t)}{1+\int_0^t\hat{u}_m(\tau)d\tau}
\leq \frac{C_3u_m(t)+C_2\sum\limits_{l=1}^{m-1}\hat{u}_l(t)}{1+\int_0^t\hat{u}_m(\tau)d\tau}+C_2C_3u_0(t).
\end{align}
On the other hand, $ 1+\int_0^t\hat{u}_m(\tau)d\tau\geq 1 $, hence we have
\begin{align}\label{CCA}
 \frac{C_3u_m(t)+C_2\sum\limits_{l=1}^{m-1}\hat{u}_l(t)}{1+\int_0^t\hat{u}_m(\tau)d\tau}\leq C_3u_m(t)+C_2\sum\limits_{l=1}^{m-1}\hat{u}_l(t).
\end{align}
Moreover, Replacing $ \frac{d}{dt}\ln\Big( 1+\int_0^t\hat{u}_m(\tau)d\tau\Big)= \frac{\hat{u}_m(t)}{1+\int_0^t\hat{u}_m(\tau)d\tau}$ in (\ref{CCC}) and using (\ref{CCA}), we conclude that
\begin{align}\label{DDD}
\frac{d}{dt}\ln\Big( 1+\int_0^t\hat{u}_m(\tau)d\tau\Big)
\leq C_3u_m(t)+C_2\sum\limits_{l=1}^{m-1}\hat{u}_l(t)+C_2C_3u_0(t),
\end{align}
for all $ t\in[0,T) $. Integration of two sides of (\ref{DDD}), yields
\begin{align*}
\ln\Big( 1+\int_0^t\hat{u}_m(\tau)d\tau\Big)&\leq C_3\int_0^tu_m(\tau)d\tau+C_2\sum\limits_{l=1}^{m-1}\int_0^t\hat{u}_l(\tau)d\tau
+C_2C_3\int_0^tu_0(\tau)d\tau\\&\leq C_3\int_0^Tu_m(\tau)d\tau+C_2\sum\limits_{l=1}^{m-1}\int_0^T\hat{u}_l(\tau)d\tau
+C_2C_3\int_0^Tu_0(\tau)d\tau.
\end{align*}
By assumption, we have $\int_0^Tu_0(\tau)d\tau<\infty$ and $\int_0^Tu_m(\tau)d\tau<\infty$. Moreover, the induction hypothesis implies that $\int_0^T\hat{u}_l(\tau)d\tau<\infty$ for $l=1,2,...,m-1$. Therefore, there exists a positive constant $ D $, such that
\begin{align*}
\ln\Big( 1+\int_0^t\hat{u}_m(\tau)d\tau\Big)&\leq D.
\end{align*}
 Finally, we conclude that
 \begin{align*}
 \int_0^T\hat{u}_m(\tau)d\tau\leq \exp(D)-1<\infty,
 \end{align*}
 as we have claimed.
\end{proof}
\end{lem}
\begin{defn}
\cite{HA}Let $ E$ be a vector bundle over a manifold $M$, and $ \Omega\subseteq M $ an open set with compact closure $ \bar{\Omega} $ in $ M $. We say that a sequence of sections $ \{\xi_k\} $ of $E$ converges in $C^p$ to $ \xi_{\infty}\in \Gamma(E|_{\bar{\Omega}}) $, for any $p\geq 0$ if for every $\epsilon>0  $ there exists $k_0 = k_0(\epsilon)$ such that
$$\sup\limits_{0\leq\alpha\leq p}\sup\limits_{x\in \bar{\Omega}}|D^{\alpha}(\xi_k-\xi_{\infty})|<\epsilon,$$
whenever, $ k> k_0 $. We say that $\{\xi_k\} $ converges in $ C^{\infty} $ to $ \xi_{\infty} $ on $ \bar{\Omega} $ if $ \{\xi_k\} $ converges in $ C^{p} $ to $ \xi_{\infty} $ on $ \bar{\Omega} $ for every $ p\in \mathbb{N} $.
\end{defn}
\begin{thm}\label{A.5}
Let $ (M,g_0) $ be a compact Finslerian manifold and $ g(t) $ solutions of the evolution equation (\ref{evolving}). If $\int_0^Tu_m(t)dt<\infty$, for $ m=0,1,2,... $, then $g(t)$ converge in $C^{\infty}$ to a smooth limit Finslerian metric $ \bar{g} $ whenever $ t $ approaches to $ T $.
\begin{proof}
We have $ \int_0^Tu_0(t)dt<\infty $, so $ \int_0^T|\omega(t)|_{_{g(t)}}dt<\infty $. By means of Lemma \ref{A.1}, there exists a positive constant $ C $, such that $ |\omega(t)|_{_{g(0)}}\leq C|\omega(t)|_{_{g(t)}} $. So we have $ \int_0^T|\omega(t)|_{_{g(0)}}dt<\infty $. Now we are in a position to prove that the metrics $g(t)$ converge to a $0$-homogenous symmetric $ (0,2)- $tensor $ \bar{g}:=g(0)+\int_0^T\omega(\tau)d\tau $, whenever $ t $ approaches to $ T $. To this end we show $\lim\limits_{t\rightarrow T} g(t)=\bar{g} $, that is,
\begin{align}\label{Defhad}
\forall \epsilon>0,\quad \exists \delta>0,\ \ \textrm{if}\ \ |t-T|<\delta \ \ \textrm{then}\ \ \sup\limits_{SM} |g(t)-\bar{g}|_{g_0}<\epsilon ,
\end{align}
where $ |.|_{g_0} $ is the norm with respect to the $ g_0 $.
 By definition of $ \bar{g} $ and (\ref{Evol}) we have
\begin{align}\label{had}
 |g(t)-\bar{g}|_{_{g(0)}}&= |(g(t)-g(0))-\int_0^T\omega(\tau)d\tau)|_{_{g(0)}}\nonumber\\&=|\int_0^t\omega(\tau)d\tau-\int_0^T\omega(\tau)d\tau|_{_{g(0)}}\leq \int_t^T|\omega(\tau)|_{_{g(0)}}d\tau\nonumber\\&\leq \int_t^T\sup\limits_{SM}|\omega(\tau)|_{_{g(0)}}d\tau.
\end{align}
On the other hand, by the mean value theorem for integrals, there exists a constant $ c $, $t<c<T $, such that $ \int_t^T\sup\limits_{SM}|\omega(\tau)|_{_{g_0}}d\tau=\sup\limits_{SM}|\omega(c)|_{_{g_0}}(T-t) $. By means of (\ref{had}) we have
$$ \sup\limits_{SM} |g(t)-\bar{g}|_{_{g_0}}\leq \sup\limits_{SM}|\omega(c)|_{_{g_0}}(T-t)=\sup\limits_{SM}|\omega(c)|_{_{g_0}}|t-T|.$$
Assuming $\delta<\frac{\epsilon}{\sup\limits_{SM}|\omega(c)|_{g_0}}$ we obtain (\ref{Defhad}). Similarly we prove $ \lim\limits_{t\rightarrow T}D^mg(t)=D^m\bar{g} $.
 By Lemma \ref{Last Lemma}, $ m^{th} $ order derivative of relation (\ref{Evol}) and Definition \ref{Def;sup}, we have
\begin{align*}
&|D^mg(t)-D^m\bar{g}|_{_{g_0}}=|D^mg(t)-\int_0^TD^m\omega(\tau)d\tau|_{_{g_0}}\\&=|\int_0^tD^m\omega(\tau)d\tau-\int_0^TD^m\omega(\tau)d\tau|_{_{g_0}}\leq \int_t^T|D^m\omega(\tau)|_{_{g_0}}d\tau\\&\leq \int_t^T\hat{u}_m(\tau)d\tau.
\end{align*}
Again by the mean value theorem for integrals there is a constant $ t<d<T $ for which
\begin{align*}
\int_t^T\hat{u}_m(\tau)d\tau=\hat{u}_m(d)(T-t)=\hat{u}_m(d)|t-T|,
\end{align*}
Assuming $ \delta<\frac{\epsilon}{\hat{u}_m(d)} $ we conclude that the $ m^{th} $ covariant derivative $ D^mg(t) $ converge to $ D^m\bar{g} $ whenever $ t $ approaches to $ T $. Finally, $g(t)$ converges in $C^{\infty}$ to a symmetric $ (0,2)- $tensor field $ \bar{g} $, whenever $ t $ approaches to $ T $. Moreover, we prove that $ \bar{g} $ is positive definite. Lemma \ref{A.1}, shows that there exists a positive constant $ B $, such that
$$0\leq\frac{1}{B}{g}_{ij}(0)v^iv^j\leq g_{ij}(t)v^iv^j\leq B{g}_{ij}(0)v^iv^j, \qquad \forall t\in [0,T).$$
Since $g(t)$ converges to a symmetric $ (0,2)- $tensor field $ \bar{g} $ whenever $ t $ approaches to $ T $, we have
$$0\leq\frac{1}{B}{g}_{ij}(0)v^iv^j\leq \bar{g}_{ij}v^iv^j\leq B{g}_{ij}(0)v^iv^j.$$
Next, assume that $ \bar{g}_{ij}v^iv^j=0 $. The last inequality yields $ {g}_{ij}(0)v^iv^j=0 $ and by positive definiteness of  $ g_{ij}(0) $ we get $ v=0 $. Therefore $ \bar{g} $ is positive definite. Next, we find a Finsler structure $ \bar{F} $ on $ M $ such that
$$\bar{g}_{kl}=\frac{1}{2}\frac{\partial^2\bar{F}^2}{\partial y^k\partial y^l}.$$
 For this purpose, multiplying $ y^i $ and $ y^j $ to the definition of $\bar g_{ij}$ yields
$$y^iy^j\bar{g}_{ij}=y^iy^jg_{ij}(0)+\int_0^Ty^iy^j\omega_{ij}(\tau)d\tau. $$
By means of  $ y^iy^jg_{ij}(0)=F^2(0) $ we get
\begin{align}\label{MEvol}
y^iy^j\bar{g}_{ij}=F^2(0)+\int_0^Ty^iy^j\omega_{ij}(\tau)d\tau.
\end{align}
By positive definiteness  of $\bar{g}_{ij}$ we put $ \bar{F}=(y^iy^j\bar{g}_{ij})^{\frac{1}{2}} $, which is $C^\infty$ in $TM_0$, by definition of $\bar g$.  Twice vertical derivatives  of  (\ref{MEvol}) yields
\begin{align}\label{Der}
\frac{1}{2}\frac{\partial^2\bar{F}^2}{\partial y^k\partial y^l}=g_{kl}(0)+\frac{1}{2}\int_0^T\frac{\partial^2}{\partial y^k\partial y^l}(y^iy^j\omega_{ij}(\tau))d\tau.
\end{align}
On the other hand, by straight forward calculation we have
\begin{align}\label{Der1}
\frac{1}{2}\frac{\partial^2}{\partial y^k\partial y^l}(y^iy^j\omega_{ij}(\tau))=\frac{1}{2}\frac{\partial^2\omega_{ij}(\tau)}{\partial y^k\partial y^l}y^iy^j+(\frac{\partial\omega_{ik}(\tau)}{\partial y^l}-\frac{\partial\omega_{il}(\tau)}{\partial y^k})y^i+\omega_{kl}(\tau),
\end{align}
for all $ \tau\in [0,T) $. Using (\ref{evolving}) we obtain
$$ \frac{1}{2}\frac{\partial^2\omega_{ij}(\tau)}{\partial y^k\partial y^l}y^iy^j=0,\qquad \frac{\partial\omega_{ik}(\tau)}{\partial y^l}y^i=0, \qquad  \frac{\partial\omega_{il}(\tau)}{\partial y^k}y^i=0.$$
Therefore (\ref{Der1}) reduces to
\begin{align}\label{Der2}
\frac{1}{2}\frac{\partial^2}{\partial y^k\partial y^l}(y^iy^j\omega_{ij}(\tau))=\omega_{kl}(\tau),
\end{align}
for all $ \tau\in [0,T) $.
Finally, replacing (\ref{Der2}) in (\ref{Der}) we get the $0$-homogenous Finslerian metric
\begin{align*}
\frac{1}{2}\frac{\partial^2\bar{F}^2}{\partial y^k\partial y^l}=g_{kl}(0)+\int_0^T\omega_{kl}(\tau)d\tau=\bar{g}_{kl}.
\end{align*}
As we have claimed.
\end{proof}
\end{thm}
\begin{cor}
Let $ (M,g_0) $ be a compact Finslerian manifold and $ g(t) $ solutions of the Ricci flow (\ref{FRF}). If $\int_0^T\sup\limits_{SM}|\nabla^m Ric_{g(t)}|_{g(t)}dt<\infty $, for all non-negative integers $m$, then $g(t)$ converges in $C^{\infty}$ to a smooth limit Finslerian metric $ \bar{g} $ whenever $ t $ approaches to $ T $.
\end{cor}
\bpf
 In the equation (\ref{Def;sup}),  $\omega(t)$ is a family of $0$-homogenous symmetric $(0,2)$-tensors on $ \pi^*TM $. Replacing  $ \omega(t) $ by
 $-2Ric_{g(t)}$  leads to
$$\int_0^Tu_m(t)dt=\int_0^T\sup\limits_{SM}|\nabla^m\omega(t)|_{g(t)}dt=2\int_0^T\sup\limits_{SM}|\nabla^m Ric_{g(t)}|_{g(t)}dt.$$
Since $\int_0^T\sup\limits_{SM}|\nabla^m Ric_{g(t)}|_{g(t)}dt<\infty$, we have $\int_0^Tu_m(t)dt<\infty$, for $ m=0,1,2,\cdots$. By Theorem \ref{A.5}, $g(t)$ converge in $C^{\infty}$ to a smooth limit Finslerian metric $ \bar{g} $ whenever $t$ approaches to $T$. This completes the proof.
\epf



\begin{thebibliography}{9}
\bibitem{AZ} H. Akbar-Zadeh, \textit{Sur les espaces de Finsler a courbures sectionnelles
constantes}, \emph{Acad. Roy. Belg. Bull. Cl. Sci}. \textbf{74} (1988), no. 5, 199-202.


\bibitem{Ak3} H. Akbar-Zadeh, \textit{Initiation to global Finslerian geometry}, vol. 68. Elsevier Science, 2006.

\bibitem{Ba} D. Bao, \textit{On two curvature-driven problems in Riemann-Finsler geometry},  Adv. stud. pure Math. \textbf{48} (2007), 19-71.

\bibitem{bcs}D. Bao, S.S. Chern and Z. Shen, \textit{An Introduction to Riemann-Finsler Geometry}.
 Graduate Texts in Mathematics, vol. \textbf{200}, Springer , 2000.


\bibitem{Br}S. Brendle, \emph{Ricci flow and the sphere theorem}. AMS, Providence, Rhode Island (2010).

\bibitem{chow} B. Chow, D. Knopf, \textit{The Ricci flow: An Introduction.
Mathematical Surveys and Monographs}, vol. \textbf{110}, American Mathematical Society, 2004.

\bibitem{Ha}  R.S. Hamilton, \textit{Three-manifolds with positive Ricci curvature}, J. Differential Geom. \textbf{17} (1989), no. 2, 255-306.

\bibitem{HA}  C. Hopper, B. Andrews, \textit{The Ricci flow in
Riemannian geometry}, Springer, 2010.
\bibitem{pre} G. Perelman, \textit{The entropy formula for the Ricci flow and its geometric applications}, preprint, http://arxiv.org/abs/math.DG/02111159.

\bibitem{Sesum}  N. Sesum, \textit{Curvature tensor under the Ricci flow}, Amer. J. math. \textbf{127}, 1315-1324 (2005).

\bibitem{Sh3} Z. Shen, {\it Lectures on Finsler Geometry}, World
 Scientific, 2001.

\bibitem{Shi}  W.X. Shi, \textit{Deforming the metric on complete Riemannian manifolds}, J. Differential Geom. \textbf{30}, 223-301 (1989).

\end{thebibliography}
\end{document}